\documentclass[12pt, reqno]{amsart}

\usepackage[top=30truemm,bottom=30truemm,left=25truemm,right=25truemm]{geometry}
\usepackage{amsmath, amssymb}
\usepackage{type1cm}
\usepackage{ascmac}
\usepackage{fancybox}
\usepackage{setspace}
\usepackage{amsthm}
\usepackage{amscd}
\usepackage{graphicx}
\usepackage{color}
\usepackage{mathrsfs}

\newtheorem{thm}{Theorem}[section]
\newtheorem{defi}[thm]{Definition}
\newtheorem{lem}[thm]{Lemma}
\newtheorem{rmk}[thm]{Remark}
\newtheorem{prop}[thm]{Proposition}
\newtheorem{cor}[thm]{Corollary}
\newtheorem{example}[thm]{Example}
\newtheorem{nota}[thm]{Notation}

\newcommand{\N}{\mathbb{N}}
\newcommand{\Z}{\mathbb{Z}}

\newcommand{\R}{\mathbb{R}}
\newcommand{\C}{\mathbb{C}}

\renewcommand{\S}{\mathbb{S}}
\newcommand{\T}{\mathbb{T}}
\newcommand{\tr}{\mathrm{\,tr\,}}
\newcommand{\g}{\mathfrak{g}}
\newcommand{\const}{\mathrm{\,const\,}}
\newcommand{\Hom}{\mathrm{\,Hom\,}}
\newcommand{\id}{\mathrm{\,id\,}}
\newcommand{\hsp}{\hspace{10mm}}
\newcommand{\D}{\mathcal{D}}
\newcommand{\ch}{ch}

\newcommand{\Ker}{{\rm Ker}\,}

\newcommand{\e}{\varepsilon}
\newcommand{\del}{\partial}
\renewcommand{\phi}{\varphi}

\newcommand{\beq}{\begin{eqnarray}}
\newcommand{\eeq}{\end{eqnarray}}
\newcommand{\beqs}{\begin{eqnarray*}}
\newcommand{\eeqs}{\end{eqnarray*}}

\begin{document}

\title{The Heat Operator of a Transversally Elliptic Operator}
\author[M. Morimoto]{Masahiro Morimoto}
\address[M. Morimoto]{Department of Mathematics, Faculty of Science, Osaka City University}
\email{d17sa001@uv.osaka-cu.ac.jp}
\date{}


\begin{abstract}
Let $G$ be a connected compact Lie group. We study the heat operator of a transversally elliptic operator. After we review the spectral properties of a transversally elliptic operator, we define the \emph{character}, that is a distribution on $G$ generalizing the trace of the heat operator to the $G$-equivariant case. The main theorem of this paper gives the estimate of $f_\alpha(t)$, which essentially determines the convergence of the \emph{character}.

\end{abstract}
\maketitle
\section*{Introduction}
In \cite{Ati74}, M. F. Atiyah extended the index theory of elliptic operators to that of  \emph{transversally elliptic} ones. These are $G$-invariant differential operators on a compact $G$-manifold which is elliptic in a direction normal to $G$-orbits, where $G$ is a compact Lie group. Although the kernel of a transversally elliptic operator $A$ is not finite dimensional like that of elliptic ones, it turns out that the character of the  representation on $\Ker (A)$ is well-defined as a distribution on $G$. Then the index of $A$ is defined by ${\rm ind}(A):= \ch[\Ker(A)]-\ch [\Ker(A^*)]$.
An explicit index formula is given in \cite{Ati74} for torus acting with only finite isotropy groups.


After that, M. A. Shubin investigated in \cite{Shu84} the spectral properties of a transversally elliptic operator $A$ and 
showed the existence of orthonormal basis consisting of its eigenfunctions. Moreover he extended the result about the distributinal character in the Atiyah's work:  the character of  the induced representation on each eigenspace $V_{\lambda}:=\Ker(A-\lambda \id)$ is well-defined as a distribution on $G$.

In this paper we study, by using Shubin's results above, the following formal series consisting of distributions on $G$: 
\beq \label{dodc}
\sum_{\lambda} e^{-t\lambda} \ch(V_{\lambda}),
\eeq
called the \emph{character of} $e^{-tA}$. Notice that if $G=\{e\}$, then $A$ is elliptic and (1) is 
\beq \label{dodc2}
\sum_{\lambda} e^{-t\lambda} \dim(V_{\lambda}).
\eeq
This is the trace of the heat operator $e^{-tA}$, which is known as an important tool for studying geometry and topology \cite{JR}. The purpose of this paper is to investigate the convergence of (\ref{dodc}). To do that, we give a survey of Shubin's paper \cite{Shu84} in the first three sections of this paper. After that, we formulate and investigate the series (\ref{dodc}) in section 4. The results of this paper are stated in  Proposition \ref{example} and Theorem \ref{Main}. The former is in a simple but a nontrivial case, and the latter is in a general case. 

\section{Review of the spectral properties of an elliptic operator}
Throughout this paper manifolds are assumed to be of class $C^\infty$ and without boundary, unless mentioned otherwise.

Let $E$ be a complex vector bundle of rank $r$ over a manifold $M$. We denote by $C^\infty(E)$ the set of smooth sections of $E$.  A linear map $P:C^\infty(E) \rightarrow C^\infty(E)$ is called a linear differential operator of order $l \in \Z_{\geq 0}$ if it is locally written by
$$
\sum_{0\leq[\alpha]\leq l} a_\alpha D^\alpha.
$$
Here
$D^\alpha:=(-i)^{[\alpha]} \prod_{i=1}^n \del^{\alpha_i}_{x_i}$ and $a_\alpha$ is an $\mathrm{End} (\C^r)$-valued smooth function on the local chart $(U,x_i)$ in $M$ for each multi-index
$
\alpha=(\alpha_1, \cdots ,\alpha_n) \in (\Z_{\geq0})^n.
$
The principal symbol $\sigma(P)$ of $P$ is an element of $\mathrm{End}\,(\pi^*E)$ locally defined by  
$$
\sigma(P)_\xi=
\sum_{[\alpha]= l} a_\alpha (p) \prod_{i=1}^n \xi_i^{\alpha_i},\hsp \xi=(\xi_1,\cdots, \xi_n) \in T^*_pM,
$$
where $\pi:T^*M\rightarrow M$ is the cotangent bundle. An equivalent definition is that  $\sigma(P)_\xi$ is the coefficient of $u^l$ in the operator 
\beq \label{symbol op}
e^{-iuf} P e^{iuf} :C^\infty(E) \rightarrow C^\infty(E),
\eeq
where $f \in C^\infty (M)$ such that $df_p=\xi$. A linear differential operator $P$ is said to be elliptic if its principal symbol $\sigma(P)$ is invertible on $T^*M\backslash 0$. 

Now we suppose that $M$ is compact and oriented. We choose and fix a Riemannian metric on $M$ and a Hermitian metric on $E$. The induced metric $\langle \, \cdot \,, \,\cdot\, \rangle_{L^2}$ on $C^\infty(E)$ is given by 
$$
\langle \phi, \psi \rangle_{L^2} = \int_M  \langle \phi(p), \psi(p) \rangle_{E_p} dM(p), \hsp \phi, \psi \in C^\infty(E),
$$
where $dM$ is the volume form on $M$ induced by its metric and orientation. A linear differential operator $P$ is called nonnegative if 
$$
\langle P\phi, \phi \rangle_{L^2} \geq 0,\hsp \phi \in C^\infty(E),
$$
and formally self-adjoint if 
$$
\langle P\phi, \psi \rangle_{L^2}
=
\langle \phi, P\psi \rangle_{L^2},
\hsp
\phi, \psi \in C^\infty(E).
$$

We say $\lambda \in \C$ to be an eigenvalue of a linear differential operator $P$ if there is $\phi \in C^\infty(E)\backslash \{0\}$ such that $P\phi =\lambda \phi$. Then the corresponding eigenspace is $V_\lambda:=\{\phi \in C^\infty(E) \ |\ P \phi=\lambda \phi  \}$. 
The following are the spectral properties of an elliptic operator. Details for (i), (ii), (iii) are in \cite{FW}, p.254. For (iv), see \cite{Hei79}, p.169 or \cite{JR}, Theorem 8.16.  
\begin{lem}\label{EO}
Let $E$ be a Hermitian vector bundle over an oriented compact Riemannian manifold $M$, and $P:C^\infty(E)\rightarrow C^\infty(E)$ be a formally self-adjoint elliptic operator of order $2m$. We assume that $P$ is nonnegative and its principal symbol $\sigma(P)$ is positive definite on  the unit cotangent bundle $T^*_1 M$. Then
\begin{enumerate}
\item $P$ has eigenvalues 
$0 \leq \lambda_1 < \lambda_2 < \cdots$ 
tending to infinity ,
\item each eigenspace $V_{\lambda_k}$ has finite dimension,
\item there is an orthonormal basis $\{\psi_j\}_{j \in \N}$ of $L^2(E)$ consisting of eigenfunctions of $P$,
\item the Weyl's asymptotic formula holds: 
\begin{eqnarray}\label{WAF}
N(t):=\sum_{\lambda\leq t} \dim V_\lambda \sim C_1\,  t^{\frac{n}{2m}},\hsp t\rightarrow\infty, 
\end{eqnarray}
where $n=\dim M$ and $C_1>0$ is given by a certain integral over the unit cotangent bundle involving the symbol of $P$.
\end{enumerate}
\end{lem}

Let $P$ be an elliptic operator in Lemma \ref{EO} and $\{\psi_j\}_{j \in \N}$ be the orthonormal basis of $L^2(E)$ consisting of  eigenfunctions of $P$. For each $j$, the corresponding eigenvalue is denoted by $\widehat{\lambda_j}$ throughout this paper. By Lemma\,\ref{EO}\,(i) and (ii), we can reorder $\{\psi_j\}_{j \in \N}$ so that 
$$
0 \leq \widehat {\lambda_1} \leq \widehat {\lambda_2} \leq \cdots \leq \widehat {\lambda_j} \leq \cdots.
$$
With respect to $\{\widehat{\lambda_j}\}_{j \in \N}$, we have the following asymptotic formula.
\begin{cor}\label{WAF3}
\begin{eqnarray}\label{WAF2}
\widehat{\lambda_j} \sim  C_2 \  j^\frac{2m}{n}, \hspace{10mm} j \rightarrow \infty
\end{eqnarray}
\end{cor}

\begin{proof}
If we put $t=\widehat{\lambda_j}$ in (\ref{WAF}), we have
$$
j \sim C_1 \  \widehat{\lambda_j}^{\frac{n}{2m}}, \hspace{10mm} j \rightarrow \infty
$$
Since $\widehat{\lambda_j} \geq 0$, if we put $C_2:=C_1^{-\frac{2m}{n}}$ then we have ($\ref{WAF2}$), as asserted.
\end{proof}
\section{The character of an infinite dimensional representation of a compact Lie group}\label{sec:1}

Let $G$ be a compact Lie group and $H$ be a separable complex Hilbert space. A group homomorphism $\pi : G \rightarrow GL_{\C}(H)$ is said to be a strongly continuous unitary representation of $G$ on $H$, or a representation of $G$ for short, if for each $g \in G$ the corresponding operator $\pi(g)$ is unitary and the map $ G \rightarrow H$, $g \mapsto \pi(g) v$ is continuous for each $v \in H$. We denote such a representation briefly by $(\pi, H)$ or $\pi$. If $\dim H < \infty$, its character is a smooth function on $G$ defined by $g \mapsto \tr \pi(g)$. Two representations $(\pi, H)$ and $(\pi', H')$ are equivalent if there exists a unitary operator from $H$ to $H'$ compatible with the $G$-actions. A representation $(\pi, H)$ is said to be irreducible if there is no nontrivial closed subspace of $H$ invariant under $G$. We write $\widehat{G}$ for the set of equivalence classes of  irreducible representations of $G$. It is known that each irreducible representation of $G$ is finite dimensional and the set $\widehat{G}$ is at most  countable. The following is a well-known lemma by Schur. For a proof, for example, see \cite{KO}, Theorem 1.42.

\begin{lem}[Schur's lemma]\label{SL1}
For $(\tau, L)$, $(\tau', L') \in \widehat{G}$, 
\beq \label{Schur eq.}
\Hom_G(L,L')
=
\begin{cases}
0\hspace{10mm}&((\tau, L)\not\cong(\tau', L'))\\
\C &((\tau, L)\underset{}{\cong}(\tau', L'))
\end{cases}.
\eeq
\end{lem}
Let $(\pi, H)$ be a representation of $G$. For each $(\tau, L) \in \widehat{G}$, we set 
$$
n_{(\tau, L)}:=\dim[\Hom_G(L,H)], \hsp
H_{(\tau,L)}:=\overline{
\sum_{\Phi \in \Hom_G(L,H)} {\rm Image}(\Phi)
}.
$$
Here the summation denotes the sum vector space in $H$ and the overline is the completion. $n_{(\tau, L)}$ is called the multiplicity of $(\tau, L)$ and $H_{(\tau, L)}$ is the $(\tau, L)$-component in $H$.  If $\dim H< \infty$, it follows from Lemma \ref{SL1} that we can write 
\beq\label{DCP1}
H_{(\tau,L)}\cong \underbrace{L\oplus \cdots \oplus L}_{n_{(\tau,L)}\ {\rm copies}}.
\eeq

Let $C(G)$ denote the set of continuous functions on $G$. For each $(\tau, L) \in \widehat{G}$, we define a map 
$
\Psi_{(\tau,L)}: L\otimes L^* \rightarrow C(G)
$ 
by 
$$
[\Psi_{(\tau,L)} (v \otimes \lambda)]_g
:=
\lambda(\tau(g^{-1})v),\hspace{10mm} v\in L,\  \lambda \in L^*,
$$
and set $\Psi:= \bigoplus_{(\tau,L)} \Psi_{(\tau,L)}$. Notice that then the image 
$
\Psi_{(\tau,L)}(L\otimes L^*) 
$
is generated by the matrix elements of $(\tau, L)$. The following Peter-Weyl theorem states that the space spanned by matrix elements of all irreducible representations is dense in $C(G)$.
For a proof, we refer the reader to \cite{FW}, p257.

\begin{thm}[Peter-Weyl]\label{pwt0}
The space $
\Psi
\left(
\bigoplus_{(\tau,L)\in \widehat{G}}L\otimes L^*
\right)
$
is dense in $C(G)$ with respect to the uniform norm.
\end{thm}

Let $dg$ be a normalized Haar measure on $G$. We define the Hermitian inner product on $C(G)$ by 
$$
(u, v)_{L^2}=\int_Gu(g)\overline{v(g)}dg, \hspace{10mm}  u,v \in C(G)
$$
and denote its completion by $L^2(G)$. The following orthogonality relation is well-known. For a proof, for instance, see \cite{KO}, Theorem 3.33.

\begin{lem}[Schur's orthogonality relations]\label{SL2}
Let $(\tau, L)$, $(\tau', L') \in \widehat{G}$. 
For each $u$, $v \in L$ and $u'$, $v' \in L'$, 
\beq \label{SOL0}
(
\langle \pi(g)u,v \rangle_{L}, \langle \pi'(g) u', v' \rangle_{L'}
)
_{L^2}
=
\begin{cases}
\hspace{7mm}0&((\tau, L)\not\cong(\tau', L'))\\
\frac{\langle u,u'\rangle \langle v,v'\rangle}{\dim L} &((\tau, L)\underset{}{\cong}(\tau', L'))
\end{cases}.
\eeq
\end{lem}

The \emph{regular representation of $G$} is an action of $G \times G$ on $L^2(G)$ given by
$$
[(g,h)\cdot \varphi](x)=\varphi(g^{-1}xh),\hsp g,h,x \in G,\ \varphi \in L^2(G).
$$
We will also consider the \emph{left regular representation of} $G$, an action of $G$ on $L^2(G)$ defined by 
$$
[g \cdot \varphi](x)=\varphi(g^{-1}x).
$$
The following is a consequence of Theorem \ref{pwt0}, which 
gives the decomposition of the regular representation  into the irreducible ones. 

\begin{cor}\label{PWT}
A map
$$
\widetilde{\Psi} :
\overline {\bigoplus_{(\tau, L) \in \widehat{G}} L \otimes L^*}
\longrightarrow 
L^2(G),
$$
defined by $\widetilde{\Psi}:= \bigoplus_{(\tau,L) \in \widehat{G}} \sqrt{\dim L}\  \Psi_{(\tau,L)}$ is a $G\times G$-equivariant unitary operator.
\end{cor}

\begin{proof}
It is easy to check that $\widetilde{\Psi}$ is $G\times G$-equivariant and isometry with respect to the induced $G\times G$-action and metric on $\bigoplus_{(\tau, L) \in \widehat{G}} L \otimes L^*$. For surjectiveness, consider the $L^2$-completion of 
$\widetilde{\Psi}
(
\bigoplus_{(\tau,L)\in \widehat{G}}L\otimes L^*
).
$
It follows from the uniform completeness in Theorem \ref{pwt0} that it is also complete with respect to the $L^2$-norm. Therefore $\widetilde{\Psi}$ is a unitary operator commuting with the $G\times G$-action.
\end{proof}

Now we set
$$
C(G)^{\rm{Ad}}:=\{f \in C(G) \ |\ f(g^{-1}xg)=f(x), \ \ \forall g\in G \}.
$$
An element in $C(G)^{\rm Ad}$ is called a \emph{class function on} $G$. Applying Theorem \ref{pwt0} to class functions, we have the following.
\begin{cor}\label{pwt2}
The set of characters of irreducible representations $\{\chi_{\tau} \}_{\tau \in \widehat{G}}$ is dense in $C(G)^{\rm Ad}$ with the uniform norm.
\end{cor}

For each $\tau \in \widehat{G}$, the orthogonal projection $P_{\tau}$ onto the $\tau$-component $H_\tau$  is given by
\beq \label{proj}
P_{\tau}(v):=(\dim \tau) \int_G \overline {\chi_{\tau}(g)}\pi(g) v \,dg, \hsp v \in H,
\eeq
which is $G$-equivariant (\cite{KO}, Theorem 4.18).   Another consequence of Theoremt \ref{pwt0} is the following. This enables us to decompose each representation  into the irreducible ones.
\begin{cor}\label{pwt3}
Let $(\pi,H)$ be a representation of $G$. Then $H$ can be decomposed into $\tau$-component $H_\tau$ in $H$: 
$$
H=\overline{\bigoplus_{\tau \in \widehat{G}}H_{\tau}}.
$$
\end{cor}
\begin{proof}
We follow the proof of \cite{KO}, Theorem 4.18 (iii). Suppose that $v \in H$ is orthogonal to $H_{\tau}$ for each $\tau \in \widehat{G}$, and let us show that $v=0$.  First we set
$$
f(x):=\int_G \langle \pi(x)\pi (g) v, \pi(g) v\rangle_H dg.
$$
Note that then $f \in C(G)^{\rm Ad}$ and $f(e)=\langle v,v\rangle_H$. Now for each $\tau \in \widehat{G}$
\beqs
(f, \chi_\tau)_{L^2}
&=&
\int_G 
\left( 
\int_G \langle \pi(x) \pi(g)v, \pi(g)v \rangle_H 
dg
\right) 
\overline{\chi_\tau(x)}
dx
\\
&=&
\int_G 
\left( 
\int_G 
\langle 
\overline{\chi_\tau(x)}
\pi(x) \pi(g)v, \pi(g)v 
\rangle_H 
dx
\right) 
dg
\\
&=&
(\dim \tau)^{-1} \int_G
\langle P_\tau (\pi(g)v), \pi(g)v \rangle_H dg
\\
&=&
(\dim \tau)^{-1} \int_G
\langle \pi(g) P_\tau v, \pi(g)v \rangle_H dg
=
(\dim \tau)^{-1} 
\langle P_\tau v, v \rangle_H =0,
\eeqs
where $P_\tau$ is a projection (\ref{proj}). By the completeness in Corollary \ref{pwt2}, we have $f=0$. Therefore $v=0$, as  claimed.
\end{proof}

Now we suppose that $G$ is connected, and choose an orientation of $G$ and a bi-invariant Riemannian metric on $G$. The corresponding Laplace-Beltrami operator on $C^\infty (G)$ is denoted by $\Delta$. Note that it can be written by $\{X_1, ...X_d\} \subset \mathfrak{g}$, an orthonormal basis of the Lie algebra,  in the form
$$
\Delta=-\sum_{j=1}^d X_j^2.
$$ 
\begin{lem}\label{ME}
For each $(\tau, L) \in \widehat{G}$, there exists an eigenvalue $\lambda_\tau$ of $\Delta$ such that
$$
\widetilde{\Psi}(L \otimes L^*) \subset V_{\lambda_\tau},
$$
where $\widetilde{\Psi}$ is an operator in Corollary \ref{PWT} and $V_{\lambda_\tau}$ is the  corresponding eigenspace of $\Delta$.
\end{lem}
\begin{proof}
Since the metric on $G$ is bi-invariant, $\Delta$ commutes with both left and right translations. Therefore it commutes with the regular representation of $G$. Since $L\otimes L^*$ is irreducible with respect to the $G\times G$-action, it follows from Lemma \ref{SL1} that there is $\lambda_\tau \in \C$ such that $\Delta|_{\widetilde{\Psi}(L\otimes L^*)}=\lambda_\tau \id$. Hence we have $\widetilde{\Psi}(L\otimes L^*) \subset V_{\lambda_\tau}$.
\end{proof}
\begin{rmk} \label{irr-eig}
The left regular representation induces a $G$-action on $L\otimes L^*$ via $\widetilde{\Psi}$. Choosing a basis of $L$, we have a $G$-isomorphism 
$$
L\otimes L^* \rightarrow \underbrace{L \oplus \cdots \oplus L}_{\dim L^*\ \mathrm{copies}}. 
$$
Via this isomorphism, we can consider the operator $\Delta|_{\widetilde{\Psi}(L)}(=\lambda_\tau \id_{\widetilde{\Psi}(L)})$. This operator will be reconsidered in a proof of Lemma \ref{DeltaG}.
\end{rmk}

\begin{nota} \label{notation}
We write $\widehat{G}=\{(\tau_\alpha, L_\alpha)\}_{\alpha \in \N}$ so that $\widetilde{\lambda}_1 \leq \widetilde{\lambda}_2\leq\cdots\leq \widetilde{\lambda}_\alpha\leq\cdots$, where $\widetilde{\lambda}_\alpha$ is the eigenvalue of $\Delta$ corresponding to the irreducible representation $(\tau_\alpha, L_\alpha)$. 
The reason why we do not use $\lambda_\alpha$ but $\widetilde{\lambda}_\alpha$ is to distinguish them with the eigenvalues of $\Delta$ taken into account of their multiplicities, i.e. $\lambda_1<\lambda_2<\cdots<\lambda_\alpha<\cdots$. For each $(\tau_\alpha, L_\alpha)$, we write $d_\alpha$ and $\chi_\alpha$ for its dimension and character, respectively. 

If we consider the basis $\{\psi_j\}_{j\in\N}$ of $L^2(G)$ consisting of eigenfunction of $\Delta$ as in Corollary \ref{WAF3}, we denote for each $\psi_j$ the corresponding eigenvalue by $\widehat{\lambda_j}$. This notation is used throughout this paper. Be careful not to confuse $\lambda_\alpha$, $\widetilde{\lambda_\alpha}$ and $\widehat{\lambda_\alpha}$.
\end{nota}
\begin{prop}\label{prop:1.1}
The series \ 
$\displaystyle 
{\sum_{\alpha=1}^{\infty} \frac{d_{\alpha}^2}{(1+\widetilde{\lambda}_\alpha)^{s}}}$
converges if and only if $s>d/2$.
\end{prop}
\begin{proof}
Let $\widehat{\lambda_1} \leq \widehat{\lambda_2} \leq \cdots \leq \widehat{\lambda_k} \leq \cdots $ be the eigenvalue of $\Delta$ not taken into account of their multiplicities. Then we have 
$$
\sum_{\alpha=1}^{\infty} \frac{d_{\alpha}^2}{(1+\widetilde{\lambda}_\alpha)^{s}}
=
\sum_{k=1}^{\infty} \frac{1}{(1+\widehat{\lambda_k})^{s}}.
$$
By the asymptotic formula (\ref{WAF2}), we have
$$
\sum_{k=1}^{\infty} \frac{1}{(1+\widehat{\lambda_k})^{s}}
=
C \sum_{k=1}^{\infty} \frac{1}{(1+k^{2/d})^{s}}
=
C' \sum_{k=1}^{\infty} \frac{1}{(1+k)^{2s/d}}.
$$
Therefore $\ \displaystyle \sum_{\alpha=1}^{\infty} \frac{d_{\alpha}^2}{(1+\widetilde{\lambda}_\alpha)^{s}}
$
converges if and only if $s>d/2$, as claimed.
\end{proof}

Throughout this paper we write $\D'(G)$ for the set of distributions on $G$. If $u \in \D'(G)$, then $u$ is a continuous linear functional on $C^\infty(G)$ and  for each $\varphi \in C^\infty(G)$ we denote the corresponding value by $\langle u, \varphi \rangle _{\D'(G)}$.
A canonical embedding $L^1(G)\hookrightarrow \D '(G)$ is given by 
$$
\langle f, \varphi \rangle _{\D'(G)}=\int_G f(g) \varphi(g) dg, \hsp f \in L^1(G),\  \varphi \in C^\infty(G).
$$

Let $(\pi, H)$ be a representation of $G$. Note that we do not assume that $\dim H$ is finite. Now let us consider the existance of its character in the distributional sense. To do this, we first consider the map 
$\widetilde{\pi}:C^\infty(G) \rightarrow B(H)$ 
defined by
\begin{eqnarray}
\widetilde{\pi} (\varphi) = \int _G \varphi(g) \pi(g) dg, \hsp \varphi \in C^\infty (G),
\end{eqnarray}
where $B(H)$ denotes the set of bounded operators on $H$. Observe that if $\dim H <\infty$, then
\begin{eqnarray}\label{eq:1.5}
\tr [\widetilde{\pi}(\varphi)]
=
\int \tr \pi(g) \varphi(g) dg
=
\int \chi(g)\varphi(g)dg
=
\langle \chi, \varphi \rangle_{\D'(G)},
\end{eqnarray}
where $\chi \in C^\infty(G) \subset \D'(G)$ is the character of $(\pi, H)$. The definition of a distributional character is based on this formula. 
\begin{thm}\label{thm:1.1}
Let $\pi : G \rightarrow GL_{\C} (H)$ be a representation of $G$.  For each $\alpha \in \N$, we denote by $n_\alpha$ the multiplicity of the irreducible representation $(\tau_\alpha, L_\alpha)$ in $H$.
Then the following conditions are equivalent:
\begin{enumerate}
\item $\widetilde{\pi}(\varphi)$ is trace-class for each $\varphi \in C^\infty(G)$,
\item $\widetilde{\pi}(\varphi) $ is trace-class for each $\varphi \in C^\infty(G)$ and the functional $\varphi \mapsto \tr  [\widetilde{\pi}(\varphi)] $ is continuous on $C^{\infty}(G)$. Hence it defines a distribution $\chi \in \mathcal {D}'(G)$ such that $\tr [\widetilde{\pi}(\varphi)] =\langle\chi,\varphi \rangle_{\D'(G)}$,
\item $n_\alpha<\infty$ for each $\alpha \in \N $ and the  series 
\begin{eqnarray}\label{eq:1.7}
\sum_{\alpha=1}^{\infty}n_\alpha \chi_\alpha
\end{eqnarray}
converges weakly to a distribution $\chi \in \mathcal{D}'(G)$,
where $\chi_\alpha$ is the character of $(\tau_\alpha, L_\alpha)$,
\item there is an $s\in \mathbb{R}$ such that 
\begin{eqnarray}\label{eq:1.8}
\sum_{\alpha=1}^{\infty} \frac{n_\alpha^2}{(1+ \widetilde{\lambda}_{\alpha})^{s}}<\infty,
\end{eqnarray}
where $ \widetilde{\lambda}_{\alpha}$ is the eigenvalue of the Laplacian $\Delta$ corresponding to $(\tau_\alpha, L_\alpha)$,
\item there are $s\in  \mathbb{R}$ and $C>0$ such that $$n_\alpha \leq C(1+ \widetilde{\lambda}_{\alpha})^s.$$ 
\end{enumerate}
Moreover, the distributions $\chi$ defined in {\rm(ii)} and {\rm(iii)} are identical. 
\end{thm}
\begin{defi}\label{defi:1.1}
If a representation $(\pi,H)$ satisfies one of the conditions in Theorem \ref{thm:1.1}, we say $(\pi,H)$ has a distributional  character and denote the distribution $\chi$ by $ch(H)$.
\end{defi}
\begin{proof}[Proof of Theorem \ref{thm:1.1}] We follow the proof of Theorem 1.1 in \cite{Shu84}.
First we assume (i) and prove (iii). We decompose
\begin{eqnarray}\label{eq:1.4}
H=\bigoplus_{\alpha \in \N} H_\alpha,
\end{eqnarray}
where $H_\alpha$ is the $(\tau_\alpha, L_\alpha)$-component in $H$. If we assume that $n_\alpha=\infty$ for some $\alpha$, we can reach a contradiction as follows.  Let $\{\lambda(\varphi)_j\}_{j=1}^{d_\alpha}$ be  the eigenvalues of $ \widetilde{\pi}(\varphi) |_{L_\alpha}$. If we regard them as the eigenvalues of $ \widetilde{\pi}(\varphi) |_{H_\alpha}$, then all the multiplicities are infinity. Since $ \widetilde{\pi}(\varphi) $ is trace-class, which cannot possess infinitely many identical nonzero eigenvalues, $\lambda(\varphi)_j$ must be zero for all $j$ and $\varphi$ . This contradicts the fact that $\pi(g)$ is unitary for all $g\in G$. Hence we can conclude that $n_\alpha$ is finite for each $\alpha$. Next we consider the convergence of the series (\ref{eq:1.7}). Since $n_\alpha$ is finite, we can write  
$$
H_\alpha = \underbrace{L_\alpha \oplus \cdots \oplus L_\alpha}_{n_\alpha\  \mathrm{copies}}
$$
for each $\alpha$. Therefore we have
$$
\tr [ \widetilde{\pi}(\varphi) |_{H_\alpha}]
=
\int_G \varphi(g) \tr \left[\pi(g)|_{H_\alpha}\right] dg
=
\int_G  n_\alpha \chi_\alpha(g) \varphi(g) dg
=
n_\alpha \langle \chi_\alpha, \varphi \rangle_{\D'(G)}
$$
for each $\varphi \in C^\infty(G)$ and $\alpha$. This shows 
$$
\sum_{\alpha=1}^{\infty}
\left|
n_\alpha \langle \chi_\alpha, \varphi \rangle_{\D'(G)}
\right|
=
\sum_{\alpha=1}^{\infty}
\left|
\tr [ \widetilde{\pi}(\varphi) |_{H_\alpha}]
\right|.
$$
Here the right term is finite by the assumption that $\widetilde{\pi}(\varphi)$ is trace-class. This shows the weak convergence of the series (\ref{eq:1.7}) in $\D'(G)$ and hence (iii) is proved.

Next we prove the equivalence of (iii), (iv) and (v). The fact that (iv) and (v) are equivalent immediately follows from Proposition \ref{prop:1.1}.  We now check the equivalence of (iii) and (iv). Suppose (iii) and set  
$$
S_N=\sum_{\alpha=1}^N n_\alpha \chi_\alpha.
$$
First we show that  $\{S_N\}_{N \in \N}$ is bounded in $L^2_{-s}(G)$ for some large $s$. To do this, it suffices to prove that in the case $G=\mathbb{T}^n$, the n-dimensional torus, for then we can generalize the results by using a partition of unity on $G$. Since $S_N$ converges weakly in $\mathcal{D}'(G)$,  it is weakly bounded, i.e., for each $\varphi \in C^\infty (G)$,
$$
\sup_N |\langle S_N, \varphi \rangle_{\D'(G)}| < \infty.
$$
If we consider $S_N$ as the bounded linear functional on $L^2(G)$, it follows from the uniform boundedness principle that  there exists a constant $C>0$ such that
$$
\sup_N \|S_N\| \leq C,
$$
where $\| \cdot \|$ denotes the operator norm. Thus the Fourier coefficients of $S_N$ are all bounded by $C$, which is independent of $N$. This shows the boundedness of  $\|S_N\|_{-s}$ for large $s$, as asserted. Furthermore, since the operator $(1+\Delta)^{-s/2}$ is an isometry from $L^2_{-s}(G)$ to $L^2(G)$, the boundedness of $S_N$ in $L^2_{-s}(G)$ is equivalent to the boundedness of 
$$
(1+\Delta)^{-s/2}S_N=\sum_{\alpha=1}^N \frac{n_\alpha}{(1+ \widetilde{\lambda}_{\alpha})^{s/2}}\ \chi_\alpha
$$
in $L^2(G)$. Moreover, since the system $\{\chi_\alpha\}_{\alpha=1}^\infty$ is orthonormal, this is equivalent to the condition that
$$
\sup_N\sum_{\alpha=1}^N \frac{n_\alpha^2}{(1+ \widetilde{\lambda}_{\alpha})^{s}}<\infty,
$$
which is same to (\ref{eq:1.8}). Therefore, we have shown that (iii) implies (iv). Conversely if we suppose (iv), then
the following series converges in $L^2(G)$:
$$
\sum_{\alpha=1}^\infty \frac{n_\alpha}{(1+ \widetilde{\lambda}_{\alpha})^{s/2}}\ \chi_\alpha.
$$
Applying the operator $(1+\Delta)^{s/2}$ to this, the series (\ref{eq:1.7}) converges in $\mathcal{D}'(G)$. Therefore (iv) is proved and hence (iii), (iv), (v) are equivalent.

We next suppose (i) and prove (ii). Recall that we have already verified that (i) implies (iii). Moreover, using the decomposition (\ref{eq:1.4}), it follows from the assumption (i) that 
$$
\tr  [\widetilde{\pi}(\varphi)] 
=
\sum_{\alpha=1}^\infty \tr  [\widetilde{\pi}(\varphi)] |_{H_\alpha}
=\sum_{\alpha=1}^\infty n_\alpha \langle \chi, \varphi \rangle_{\D'(G)} = \langle \chi, \varphi \rangle_{\D'(G)},
$$
where 
$$
\chi=\sum_{\alpha=1}^\infty n_\alpha \chi_\alpha .
$$
Therefore the functional $\varphi \mapsto \tr  [\widetilde{\pi}(\varphi)] $ coincides with the distribution $\chi \in \mathcal{D}'(G)$ in (iii). This shows that (i) implies (ii), as claimed. Consequently, (i) and (ii) are equivalent.

Finally, we prove that (iii) implies (i). From now on, $\| A \|$, $\| A \|_1=\tr \sqrt{A^*A}$, and $\| A \|_2=[\tr (A^*A)]^{1/2}$ denote respectively the operator, trace, and Hilbert-Schmidt norm on $H$, where $A$ is an operator on $H$. Recall that the following relation holds between them:
\begin{eqnarray}\label{eq:1.9}
\| A \| \leq \| A \|_2 \leq \| A \|_1 \leq d\| A \|, \hsp d=\dim H.
\end{eqnarray}
To show (i), it suffices to prove that $\| \widetilde{\pi}(\varphi)\|_1$ is finite for each $\varphi \in C^\infty(G) $. In view of the decomposition (\ref{eq:1.4}) we can write
\begin{eqnarray}\label{eq:1.10}
\|  \widetilde{\pi}(\varphi)  \|_1=\sum_{\alpha=1}^\infty \| \widetilde{\pi}(\varphi) |_{H_\alpha}\|_1.
\end{eqnarray}
By the inequality (\ref{eq:1.9}), we have
\begin{eqnarray}\label{eq:1.11}
\| 
\widetilde{\pi}(\varphi) |_{H_\alpha}
\|_1 
\leq
d_\alpha n_\alpha 
\| \widetilde{\pi}(\varphi) |_{H_\alpha}\|.
\end{eqnarray}Using the inequality (\ref{eq:1.9}) again, we have
\begin{eqnarray}\label{eq:1.11.2}
\| 
 \widetilde{\pi}(\varphi) |_{H_\alpha} 
\| 
= 
\| 
 \widetilde{\pi}(\varphi) |_{L_\alpha} 
\| 
\leq 
\| 
\widetilde{\pi}(\varphi) |_{L_\alpha} 
\|_2 
=
\sqrt{
\sum_{i,j=1}^{d_\alpha}
| \widetilde{\pi}_{ij}^\alpha (\varphi)|^2
},
\end{eqnarray}
where $\{ \widetilde{\pi}_{ij}^\alpha (\varphi) \}_{i,j=1}^{d_\alpha}$ are the matrix elements of the operator $ \widetilde{\pi}(\varphi) $ in $L_\alpha$. Notice that using the matrix element $\pi_{ij}^{\alpha}(g)$ of the representation $(\tau_\alpha, L_\alpha)$, we can write 
\begin{eqnarray}\label{eq:1.12}
\widetilde{\pi}_{ij}^\alpha(\varphi)
=
\int_G \varphi (g) \pi_{ij}^\alpha(g)dg.
\end{eqnarray} 
By Lemma \ref{ME}, we have
$$
\int_G \varphi (g) \pi_{ij}^\alpha(g)dg
=
(1+ \widetilde{\lambda}_{\alpha})^{-t}\int_G \varphi(g) [(1+\Delta)^t\pi_{ij}^\alpha(g)]dg
$$
Integrating by parts, we have
\begin{eqnarray*}
(1+ \widetilde{\lambda}_{\alpha})^{-t}\int_G \varphi(g) [(1+\Delta)^t \pi_{ij}^\alpha(g)]dg
=
(1+ \widetilde{\lambda}_{\alpha})^{-t}\int_G[(1+\Delta)^t\varphi (g)]\pi_{ij}^\alpha(g)dg
\end{eqnarray*}
Therefore, we have
\begin{eqnarray*}
|
\widetilde{\pi}_{ij}^\alpha(\varphi)
|^2
&=&
\frac{1}{(1+ \widetilde{\lambda}_{\alpha})^{2t}}
\left|
\int_G[(1+\Delta)^t\varphi (g)]\pi_{ij}^\alpha(g)dg
\right|^2
\\
&=&
\frac{1}{(1+ \widetilde{\lambda}_{\alpha})^{2t}}
\left|
((1+\Delta)^t\varphi,\ \pi_{ij}^\alpha)_{L^2}
\right|^2
\end{eqnarray*}
By the Cauchy-Schwarz inequality, we have 
$$
\frac{1}{(1+ \widetilde{\lambda}_{\alpha})^{2t}}
\left|
((1+\Delta)^t\varphi,\ \pi_{ij}^\alpha)_{L^2}
\right|^2
\leq
\frac{1}{(1+ \widetilde{\lambda}_{\alpha})^{2t}}
\|
(1+\Delta)^t\varphi
\|_{L^2}^2
\ 
\| 
\pi_{ij}^\alpha
\|_{L^2}^2.
$$
Since 
$\varphi \in C^\infty(G)$, 
$
\|
(1+\Delta)^t\varphi
\|_{L^2}^2
$
is finite for each $\varphi$. Moreover by Lemma \ref{SL2},  we have 
$\| 
\pi_{ij}^\alpha
\|_{L^2(G)}^2
\leq
d_\alpha^{-1}. 
$  
Consequently,
$$
|\widetilde{\pi}_{ij}^\alpha(\varphi)|^2 
\leq
\frac{C_{\varphi,t}}{d_\alpha(1+ \widetilde{\lambda}_{\alpha})^{2t}}, 
$$
where $C_{\varphi,t}\geq0$ is a constant depend only on $\varphi$ and $t$. Using this estimate and (\ref{eq:1.11.2}), we have 
$$
\| 
 \widetilde{\pi}(\varphi) |_{H_\alpha} 
\|^2 
\leq
\frac{C_{\varphi,t}\ d_\alpha}{(1+ \widetilde{\lambda}_{\alpha})^{2t}}.
$$
By (\ref{eq:1.10}) and (\ref{eq:1.11}), we have 
$$
\| \widetilde{\pi}(\varphi)  \|_1 
\leq 
\sqrt{C_{\varphi, t}} 
\sum_{\alpha=1}^\infty 
\frac{n_\alpha d_\alpha^{3/2}}{(1+ \widetilde{\lambda}_{\alpha})^{t}}.
$$ 
Moreover, it follows from Proposition \ref{prop:1.1} that if $l>d/4$,
$$
d_\alpha \leq \const (1+ \widetilde{\lambda}_{\alpha})^l.
$$
Hence if we choose $t_0$ so that $t_0>3d/8$, it follows from the Cauchy-Schwarz inequality that
\begin{eqnarray*}
\| \widetilde{\pi}(\varphi) \|_1^2
&\leq& 
C'_{\varphi,t}
\left[
\sum_{\alpha=1}^\infty
\frac{n_\alpha}{(1+ \widetilde{\lambda}_{\alpha})^{(t-\frac{3}{2}t_0)}}
\right]^2
\\
&\leq& 
C'_{\varphi,t}
\left[
\sum_{\alpha=1}^\infty 
\frac{n_\alpha^2}{(1+ \widetilde{\lambda}_{\alpha})^{s}}
\right]
\left[
\sum_{\alpha=1}^\infty
\frac{1}{(1+ \widetilde{\lambda}_{\alpha})^{2(t-\frac{3}{2}t_0)-s}}
\right]
\\
&\leq& 
C''_{\varphi, t}
\sum_{\alpha=1}^\infty 
\frac{1}{(1+ \widetilde{\lambda}_{\alpha})^{2(t-\frac{3}{2}t_0)-s}},
\end{eqnarray*}
where $s\in \R$ a constant in (iv), which is equivalent to 
the assumption (iii). If we choose $t$ so that $2(t-\frac{3}{2}t_0)-s>d/2$, it follows from Proposition \ref{prop:1.1} that  the last term in the above inequalities converges. Hence $\|\widetilde{\pi}(\varphi)\|_1$ is finite for each $\varphi \in C^\infty(G)$, as asserted.
\end{proof}


\section{Spectral properties of a transversally elliptic operator}\label{thirdsec}

Let $G$ be a compact Lie group. A $G$-manifold is a manifold with a smooth $G$-action on it. A vector bundle over a $G$-manifold is called a $G$-vector bundle if it is a $G$-manifold such that the action is compatible with the action on the base space and is linear on each fiber. 

Let $E$ be a $G$-vector bundle over an oriented compact $G$-manifold $M$. We can choose a $G$-invariant Riemannian metric on $M$ and a $G$-invariant Hermitian metric on $E$, which we fix throughout this paper. Let $C^\infty (E)$ be the space of differentiable sections of $E$ with the $L^2$-metric
$$
\langle \varphi,\psi\rangle_{L^2}=\int_M \langle \varphi(p),\psi(p)\rangle_{E_p}dM(p), \hspace{10mm}\varphi, \psi \in C^\infty(E)
$$
and  denote its completion by $L^2(E)$. If we set 
$$
[g\cdot \varphi](p)
=
g[\varphi(g^{-1}p)],
\hspace{10mm}
g\in G,\  \varphi \in C^\infty(E),\ p\in M,
$$
then $G$ acts on $C^\infty(E)$ from left unitarily. 

We denote by $\g$ the Lie algebra of $G$. For each $X \in \g$, the induced differential operator $\widehat{X}:C^\infty(E)\rightarrow C^\infty(E)$ is given by
$$
[\widehat{X} \varphi]_p = \left. \frac{d}{dt} \right|_{t=0}  (\exp tX \cdot \varphi)_p, \hsp p \in M , \ \varphi \in C^\infty(E).
$$
If $E$ is a trivial line bundle $M \times \C$, then the induced  differential operator on $ C^\infty(E)=C^\infty(M)$ is  denoted by $\widetilde{X}$, called the associated vector field.
We set
$$ 
T^*_GM:=\{v\in T^*M| v(\widetilde{X}_{\pi(v)})=0 \ \mathrm{for\ each}\  X \in \g\}.
$$

\begin{defi}\label{def:2.1}
A linear differential operator $A:C^\infty(E)\rightarrow C^\infty(E)$ is called transversally elliptic if it satisfies the following two conditions:
\begin{enumerate}
\item $A$  is invariant under the $G$-action on $C^\infty(E)$, 
\item the principal symbol $\sigma(A)$ is invertible on $T^*_G M\backslash0$.
\end{enumerate}
\end{defi}
\begin{example}
If $G$ is finite or $G$ acts trivially on $M$, then a transversally elliptic operators are the $G$-invariant elliptic ones. However, in the case of a transitive action, all $G$-invariant operators are transversally elliptic.  If $G=\mathbb{S}^1$ acts on $M=\mathbb{T}^2$ in the first variable, then transversally elliptic operators are the following:
$$
\sum_{k+l \leq m} a_{kl}(y) \frac{\partial^{k+l}}{\partial x^k \partial y^l},
$$
where $a_{kl} \in C^\infty(\S^1)$ and $m$ is an order of the operator.
\end{example}
\begin{rmk} \label{rmk}
Let $A$ be a transversally elliptic operator. 
\begin{enumerate}
\item An eigenspace of $A$ can be infinite dimensional. For example, consider the operator $A=-\frac{\partial^2}{\partial y^2}$ on $\mathbb{T}^2$ with the $\S^1$-action in the example above. If $\lambda_A$ is an eigenvalue of $A$, then its eigenspace is $L^2(\S^1) \otimes W_{\lambda_A}$, where $W_{\lambda_A}$ is that of the Laplacian on $\S^1$. 

\item The eigenvalues of $A$ can be dense in $\R$. For instance, consider the operator $A=\frac{1}{i} (\frac{\partial}{\partial x}- \sqrt{2} \frac{\partial}{\partial y})$ on $\T^2$ with the $\S^1$-action above. Then its spectrum is $\{\xi_1-\sqrt{2} \xi_2 \ |\  \xi_1,\xi_2 \in \Z \}$, which is dense  in $\R$.
\end{enumerate}
Comparing with Lemma \ref{EO}, the spectral properties of a transversally elliptic operator are different from those of  elliptic operators. However, there are common properties between them, which we will see in Corollary \ref{A'sONB}.
\end{rmk}
From now on, we suppose that $G$ is connected, and  choose an orientation of $G$ and a bi-invariant Riemannian metric on $G$. Let $\{X_1, \cdots, X_d\}$ be an orthonormal basis in $\g$ and we define a linear  differential operator $\Delta_G:C^\infty(E)\rightarrow C^\infty(E)$ by 
$$
\Delta_G=-\sum_{j=1}^d \widehat{X}_j^2.
$$

\begin{lem}\label{PODG}
\begin{enumerate}
\item The principal symbol of $\Delta_G$ is 
$$
\sigma(\Delta_G)_\xi 
=
\sum_{j=1}^d [\xi (\widetilde{X_j})]^2,
\hsp \xi \in T^*_p M.
$$
\item $\Delta_G$ is nonnegative and formally self-adjoint.
\end{enumerate}
\end{lem}
\begin{proof}
We follow the proof of Lemma 2.2 in \cite{Shu84}. (i) We first calculate $\sigma(\widehat{X_j})_\xi$ for each $\xi \in T^*_p M$ using the the operator (\ref{symbol op}).  Set $g_t:=\exp tX$. For each $\varphi \in C^\infty(E)$, $p \in M$ and $f \in C^\infty(M)$ satisfying $df_p=\xi$, we have 
\beqs
[e^{-iuf(p)}(\widehat{X_j} e^{iuf} \varphi)]_p
&=&
e^{-iuf(p)}
\left. \frac{d}{dt} \right|_{t=0} g_t [e^{iuf(g_t^{-1}p)}\varphi(g_t^{-1}p)]
\\
&=&
e^{-iuf(p)}
\left. \frac{d}{dt} \right|_{t=0} e^{iuf(g_t^{-1}p)}\  g_t [\varphi(g_t^{-1}p)]
\\
&=&
iu\ 
[\widetilde{X_j} f]_p [\varphi(p)]+r
\\
&=&
iu [(df)_p(\widetilde{X_j})] [\varphi(p)]+r
=
iu [\xi(\widetilde{X}_j)] [\varphi (p)]+r,
\eeqs
where $r$ is the term not containing $u$. This shows  
$
\sigma(\widehat{X_j})_\xi 
=i \xi(\widetilde{X_j}).
$
Therefore we have 
\beqs
\sigma(\Delta_G)_\xi
=
\sigma
\left(
-\sum_{j=1}^d \widehat{X_j}^2
\right)_\xi
=
-\sum_{j=1}^d \sigma
(\widehat{X_j})^2
_\xi
=
-\sum_{j=1}^d 
[i\xi(\widetilde{X_j})]^2
=
\sum_{j=1}^d [\xi (\widetilde{X_j})]^2.
\eeqs
(ii) It follows from integration by parts that the operator $i\widehat{X_j}$ is formally self-adjoint. This shows that $\Delta_G=\sum_{j} (i \widehat{X_j})^2$ is also formally self-adjoint and nonnegative.
\end{proof}
\begin{lem} \label{DeltaG}
Let $V$ be a finite dimensional $G$-invariant subspace in $C^\infty(E)$ such that the representation on $V$ is equivalent to some $(\tau_\alpha, L_\alpha) \in \widehat{G}$. Then $\Delta_G|_V = \widetilde{\lambda_\alpha} \id_V$, where $\widetilde{\lambda_\alpha}$ is an eigenvalue of the Laplacian $\Delta: C^\infty(G)\rightarrow C^\infty(G)$ corresponding to $(\tau_\alpha, L_\alpha)$ (see Section 2, Notation \ref{notation}).
\end{lem}
\begin{proof}
Let $\widetilde{\tau} :\g \rightarrow \mathfrak{gl}_\C(V)$ be a $\g$-representation given by $X \mapsto \widehat{X}$. This representation is extended to that of the universal enveloping algebra $U(\g)$. Then $\Delta_G|_V$ is by definition identical to the endomorphism $-\sum _{i=1}^d \widetilde{\tau}(X_i)$. On the other hand, consider the differential of the representation $(\tau_\alpha, L_\alpha)$:
$$
d \tau_\alpha :\g\rightarrow \mathfrak{gl}_\C(L_\alpha).
$$
Since $d\tau_\alpha(X) = \left. \frac{d}{dt} \right|_{t=0}[\tau_\alpha (\exp tX)]$  for each $X \in \g$  (\cite{FW}, 3.36), $ \widetilde{\tau}$ and $d \tau_\alpha$ are equivalent as a $\g$-representation. Moreover, recall the map $\Delta|_{\Psi(L_\alpha)}(= \widetilde{\lambda_\alpha} \id_{L_\alpha})$, considered in Remark \ref{irr-eig}. If we denote by $\widetilde{\rho}:U(\g) \rightarrow \mathfrak{gl}_\C (\Psi(L_\alpha))$ the $U(\g)$-representation induced by the directional derivative on $G$, then  $\Delta|_{\Psi(L_\alpha)}$ is identical to $ -\sum_{i=1}^d \widetilde{\rho}(X_i)$. Since $\widetilde{\rho}$ and $d\tau_\alpha$ are equivalent as $\g$-representation, consequently we have $\Delta_G|_V = \widetilde{\lambda_\alpha} \id_V$.
\end{proof}

\begin{lem} 
Let $A:C^\infty( E )\rightarrow C^\infty( E )$ be a formally self-adjoint transversally elliptic operator of order $m>0$. Then 
\beq \label{DOP}
P:=A^2+(\Delta_G)^m
\eeq
is an elliptic operator satisfying the assumption in Lemma \ref{EO} : formally self-adjoint, nonnegative, and its principal symbol $\sigma(P)$ is positive definite on the unit cotangent bundle $T^*_1M$.  
\end{lem}
\begin{proof}
We follow the proof of Theorem 2.1 in \cite{Shu84}.
By Lemma \ref{PODG}, it is apparent that $P$ is formally self-adjoint and nonnegative. Let us check the positive definiteness of $\sigma(P)$ on $T^*_1M$. If $\xi \in T_1^*M \cap T^*_GM $, it follows from Lemma \ref{PODG} (i) that $\sigma(P)_\xi=0$. Since $\sigma(A)_\xi$ is invertible, $\sigma(A)_\xi^2$ is positive definite. Therefore $\sigma(P)_\xi=\sigma(A)_\xi^2+\sigma(\Delta_G)_\xi^m$ is invertible.  On the other hand, suppose that $\xi \in T^*_1M \backslash T^*_GM$. Then $\sigma(A)_\xi^2$ is nonnegative and $\sigma(\Delta_G)^m$ is positive definite, which show the positive definiteness of $\sigma(P)_\xi$. Hence $\sigma(P)$ is positive definite on $T^*_1M$.
\end{proof}
\begin{thm}\label{ONB}
Let $A:C^\infty( E )\rightarrow C^\infty( E )$ be a formally self-adjoint transversally elliptic operator of order $m>0$ and $P$ be the elliptic operator (\ref{DOP}). Then there is an orthonormal basis $\{ \psi_j \}_{j \in \N}$ in $L^2( E )$ consisting of joint eigenfunctions of $A$, $\Delta_G$, $P$. 
\end{thm}

\begin{cor}\label{A'sONB}
Let $A:C^\infty( E )\rightarrow C^\infty( E )$ be a formally self-adjoint transversally elliptic operator of order $m>0$. Then there is an orthonormal basis $\{\psi_j\}_{j\in \N}$ consisting of its joint eigenfunctions.
\end{cor}

\begin{proof}[Proof of Theorem \ref{ONB}] We follow the proof of Theorem 2.1 in \cite{Shu84}. Since $P$ satisfies the assumption in Lemma \ref{EO}, we can consider the following decomposition:
$$
L^2(E)=\overline{\bigoplus_{\lambda_P} {V_{\lambda_P}}},
$$
where $V_{\lambda_P}$ is the finite dimensional eigenspace of $P$. Since $A$ is $G$-invariant, it commutes with $\Delta_G$ and $P$. This shows that $V_{\lambda_P}$ is $A$-invariant and we can decompose each $V_{\lambda_P}$ by the eigenspace of $A|_{V_{\lambda_P}}$, which we denote by $W_{\lambda_A}$:
$$
V_{\lambda_P}=\bigoplus_{\lambda_A} W_{\lambda_A}.
$$
Similarly, we can also decompose $W_{\lambda_A}$ by $U_{\lambda_G}$, the eigenspaces of $\Delta_G|_{W_{\lambda_A}}$:
$$
W_{\lambda_A}=\bigoplus_{\lambda_G} {U_{\lambda_G}}.
$$
Choosing an orthonormal basis of each $U_{\lambda_G}$ and putting together them, we have the desired basis.
\end{proof}
\begin{rmk}
Let $\{\psi_j\}_{j\in\N}$ be the basis in Theorem \ref{ONB} and $\widehat{\lambda_{A,j}}$, $\widehat{\lambda_{G,j}}$, $\widehat{\lambda_{P,j}}$ be the corresponding eigenvalues of $A$, $\Delta_G$, $P$, respectively. Then the following relation holds:
\begin{eqnarray}\label{eq:2.10}
\widehat{\lambda_{P,j}}=\widehat{\lambda_{A,j}}^2+ \widehat{\lambda_{G,j}}^m, \hsp j \in \N.
\end{eqnarray}
\end{rmk}
\begin{rmk}\label{ONB3}
In the proof of Theorem \ref{ONB}, it follows from $G$-invariance of $A$, $\Delta_G$, $P$ that $U_{\lambda_G}$ is also $G$-invariant and so we can decompose it into the irreducible representations:
$$
U_{\lambda_G}=\bigoplus_{\alpha \in \N} [U_{\lambda_G}]_\alpha, 
$$
Choosing an orthonormal basis of each $[U_{\lambda_G}]_\alpha$ and putting together them, we have an orthonormal basis $\{\psi_{\alpha,j}\}_{\alpha, j }$ in $L^2(E)$ consisting of joint eigenfunctions of $A$, $\Delta_G$, $P$ such that $\{\psi_{\alpha,j}\}_{j}$ is a basis of the $(\tau_\alpha, L_\alpha)$-component $[L^2(E)]_\alpha$ for each $\alpha \in \N$.
\end{rmk}
\begin{lem}\label{EODG}
Suppose that a formally self-adjoint transversally elliptic operator  $A:C^\infty( E )\rightarrow C^\infty( E )$ of order $m>0$ is given. Then $\Delta_G|_{[L^2(E)]_\alpha}=\widetilde{\lambda_\alpha} \id$.\end{lem}
\begin{proof}
We use the notation in the proof of Theorem \ref{ONB} and Remark \ref{ONB3}. It immediately follows from Lemma \ref{DeltaG} that $\Delta_G|_{[U_{\lambda_G}]_\alpha}=\widetilde{\lambda_\alpha} \id$. Since
$$
[L^2(E)]_\alpha=\overline{\bigoplus_{\lambda_G, \lambda_A, \lambda_P}[U_{\lambda_G}]_\alpha}, 
$$
we have $\Delta_G|_{[L^2(E)]_\alpha}=\widetilde{\lambda_\alpha} \id$, as asserted.
\end{proof}
\begin{rmk} \label{ONB4}
Let $\{\psi_{\alpha,j}\}_{\alpha, j}$ be the basis considered in Remark \ref{ONB3}, and $\widehat{\lambda_{A, \alpha, j}}$, $\widehat{\lambda_{G, \alpha,j}}$, $\widehat{\lambda_{P, \alpha,j}}$ denote the corresponding eigenvalues of $A$, $\Delta_G$, $P$. By Lemma \ref{EODG}, $\widehat{\lambda_{G,\alpha,j}}=\widetilde{\lambda_\alpha}$ for all $j \in \N$ and so it follows from  (\ref{eq:2.10}) that
\beq \label{ROEV}
\widehat{\lambda_{P,\alpha,j}}=\widehat{\lambda_{A,\alpha,j}}^2+\widetilde{\lambda_\alpha}^m
, \hsp
\alpha, j \in \N
\eeq
\end{rmk}

Now we investigate each eigenspace of a transversally elliptic operator as a $G$-representation space; recall that $G$ acts on it unitarily  from left.
\begin{thm}\label{thm:2.2}
Let $A:C^\infty( E )\rightarrow C^\infty( E )$ be a formally self-adjoint transversally elliptic operator of order $m>0$. Then the 
for each $\mu\geq0$, a direct sum of the eigenspaces
\begin{eqnarray}\label{eq:2.10.3}
\overline{\bigoplus_{-\mu\leq\lambda_A\leq \mu} V_{\lambda_A}
}
\end{eqnarray}
has a distributional character $ch(\overline{\bigoplus_{-\mu\leq\lambda_A\leq \mu} V_{\lambda_A}
})$, which belongs to the Sobolev space $L^2_{-k}(G)$ where $k>n+d$. 
\end{thm}
\begin{proof}[Proof of Theorem \ref{thm:2.2}] We follow the proof of Theorem 2.2 in \cite{Shu84}. Let $n_\alpha=$\\$n_\alpha(\overline{\bigoplus_{-\mu\leq\lambda_A\leq \mu} V_{\lambda_A}
})$ be the multiplicity of the space (\ref{eq:2.10.3}). By Theorem $\ref{thm:1.1}$, it suffices to prove that $n_\alpha$ is finite for each $\alpha$ and estimated by
$$
n_\alpha\leq C(1+\widetilde{\lambda_\alpha})^s
$$
for some $s\in\mathbb{R}$ and $C>0$. Let us decompose the space (\ref{eq:2.10.3}) into the irreducible representations:
$$
\overline{\bigoplus_{-\mu\leq\lambda_A\leq \mu} V_{\lambda_A}
}=\overline{\bigoplus_{\alpha \in \N} H_\alpha},
$$
where $H_\alpha$ is the $(\tau_\alpha, L_\alpha)$-component. Let $\{\psi_{\alpha,j}\}_j$ be an orthonormal basis in $H_\alpha$ consisting of joint eigenfunctions of $A$, $\Delta_G$, $P$ (see Remark \ref{ONB3}).   For each $\psi_{\alpha,j}$, we denote the corresponding eigenvalues of $A$, $\Delta_G$, $P$ by $\widehat{\lambda_{A,\alpha, j}}$, $\widehat{\lambda_{G,\alpha, j}}$, $\widehat {\lambda_{P,\alpha,j}}$, respectively. By the assumption and (\ref{ROEV}), we have
\begin{eqnarray}\label{eq:2.10.5}
\widehat{\lambda_{P,\alpha,j}}
=
\widehat{\lambda_{A,\alpha,j}}^2+  \widetilde{\lambda_\alpha}^m 
\leq 
\mu^2+ \widetilde{\lambda_\alpha}^m
\end{eqnarray}
for each $j$ and $\alpha$. Therefore we have 
\begin{eqnarray}\label{eq:2.10.6}
H_\alpha
\subset
\bigoplus_{\lambda_P\leq\mu^2+ \widetilde{\lambda_\alpha}^m}V_{\lambda_P}.
\end{eqnarray}
Since the right term in (\ref{eq:2.10.6}) is finite dimensional, so is $H_\alpha$, which shows that $n_\alpha<\infty$. Comparing the dimension in (\ref{eq:2.10.6}), we have
\begin{eqnarray}\label{eq:2.11}
d_\alpha n_\alpha 
\leq 
\sum_{\lambda_P\leq\mu^2+  \widetilde{\lambda_\alpha}^m} \dim V_{\lambda_P}.
\end{eqnarray}
By the asymptotic formula (\ref{WAF}) in Lemma \ref{EO}, 
\begin{eqnarray}\label{eq:2.12}
d_\alpha n_\alpha \leq C_3(\mu^2+  \widetilde{\lambda_\alpha}^m)^\frac{n}{2m}\leq C_4 (1+\widetilde{\lambda_\alpha})^\frac{n}{2}.
\end{eqnarray}
Since $d_\alpha\geq1$, we have 
$$
n_\alpha \leq C_4 (1+\widetilde{\lambda_\alpha})^\frac{n}{2},
$$
which shows that $\overline{\bigoplus_{-\mu\leq\lambda_A\leq \mu} V_{\lambda_A}
}$ has a distributional character $ch(\overline{\bigoplus_{-\mu\leq\lambda_A\leq \mu} V_{\lambda_A}
})$, as claimed.

Next, we estimate the smoothness in the sense of Sobolev. Using Theorem \ref{thm:1.1}, we first expand $ch \left(\overline{\bigoplus_{-\mu\leq\lambda_A\leq \mu} V_{\lambda_A}
}\right)$  into the irreducible representations:
$$
ch \left(\overline{\bigoplus_{-\mu\leq\lambda_A\leq \mu} V_{\lambda_A}
}\right)=\sum_{\alpha=1}^\infty n_\alpha \chi_\alpha.
$$
Therefore, 
\begin{eqnarray}\label{eq:2.13}
(1+\Delta)^{-t}ch \left(\overline{\bigoplus_{-\mu\leq\lambda_A\leq \mu} V_{\lambda_A}
}\right)=\sum_{\alpha=1}^{\infty}\frac{n_\alpha \chi_\alpha}{(1+\widetilde{\lambda_\alpha})^t}.
\end{eqnarray} 
Note that since $(\tau_\alpha, L_\alpha)$ is unitary, the inequality $|\chi_\alpha(g)|\leq d_\alpha$ holds for all $g$. Hence the series in (\ref{eq:2.13}) has the majorant:
$$
\sum_{\alpha=1}^{\infty} \frac{n_\alpha d_\alpha}{(1+\widetilde{\lambda_\alpha})^t}
$$ 
for $t>0$. By (\ref{eq:2.12}), we have
$$
\sum_{\alpha=1}^{\infty} \frac{n_\alpha d_\alpha}{(1+\widetilde{\lambda_\alpha})^t}
\leq
C_4
\sum_{\alpha=1}^\infty \frac{1}{(1+\widetilde{\lambda_\alpha})^{t-\frac{n}{2}}}.
$$ 
By Proposition \ref{prop:1.1}, the right term in the above inequality converges for $t-\frac{n}{2}>\frac{d}{2}$. Hence if $2t<n+d$, we can write 
$$
f(g)=\sum_{\alpha=1}^\infty \frac{n_\alpha \chi_\alpha}{(1+\widetilde{\lambda_\alpha})^t}\in C(G) (\subset L^2(G)),
$$
since the series in the right term converges uniformly. Therefore if $k>n+d$, we have
$$
ch \left(\overline{\bigoplus_{-\mu\leq\lambda_A\leq \mu} V_{\lambda_A}
}\right)=(1+\Delta)^t f\in L^2_{-k}(G)
$$
as claimed. 
\end{proof}
The following can be  proved in the same way.
\begin{cor}\label{thm:2.2'}
For each eigenvalue $\lambda_A$, the corresponding eigenspace 
$V_{\lambda_A}$ has a distributional character $ch(V_{\lambda_A})$, which belongs to the Sobolev space $L^2_{-k}(G)$ where $k>n+d$. 
\end{cor}

\section{The heat operator of a transversally elliptic operator}

Let $M$ be an oriented compact Riemannian manifold and $\Delta_M$ be the Laplace-Beltrami operator acting on $C^\infty( E )$, where $E=\Lambda^*(T^*M)$ is the exterior bundle of $M$. Recall that the heat operator $e^{-t\Delta_M}:L^2(E)\rightarrow L^2(E)$ of $\Delta_M$ is defined by
$$
e^{-t\Delta_M}(s)
=
\sum_{\lambda}
e^{-t\lambda} s_\lambda,
$$
where $\lambda$ is an eigenvalue of $\Delta_M$ and  $s=\sum_\lambda s_\lambda$ is an $L^2$-section of $E$ expanded by the eigenfunctions of $\Delta_M$. It is known that its trace is given by 
\beq \label{tcc}
\sum_{\lambda} e^{-t\lambda}  \dim V_{\lambda}, 
\eeq
where $V_{\lambda}$ is the eigenspace of $\Delta_M$. This and its asymptotic expansion is known as an important tool for studying geometry and topology \cite{JR}. 

From now on,  as in Section 2 and 3, let $G$ be an oriented connected compact Lie group with a bi-invariant Riemannian metric. Let also $E$ be a complex $G$-bundle over an oriented compact $G$-manifold $M$ with $G$-invariant metrics. By Corollary $\ref{A'sONB}$, we can generalize the above definition to  transversally elliptic operators.
\begin{defi}
Let $A$ be a nonnegative formally self-adjoint transversally elliptic operator of order $m>0$. We define its heat operator $e^{-t A}:L^2(E)\rightarrow L^2(E)$ by
$$
e^{-t A}(s)
=
\sum_{\lambda_A}
e^{-t\lambda_A} s_{\lambda_A},
$$
where $\lambda_A$ is an eigenvalue of $A$ and  $s=\sum_{\lambda_A} s_{\lambda_A}$ is an $L^2$-section of $E$ expanded by the eigenfunctions of $A$. 
\end{defi}
The purpose of this paper is to define and investigate the  \emph{character}, that is a distribution on $G$ generalizing the trace of the heat operator to the $G$-equivariant case. Roughly speaking, the \emph{character}  is
\begin{eqnarray}\label{CH}
\sum_{\lambda_A} e^{-t\lambda_A} ch(V_{\lambda_A}),
\end{eqnarray}
where $ch(V_{\lambda_A})$ is a distributional character considered in Corollary \ref{thm:2.2'}. Notice that the classical case (\ref{tcc}) is covered by $G=\{e\}$ since 
$\{e\}$-transversally elliptic operators are elliptic ones and then  $ch(V_{\lambda_A})=\dim V_{\lambda_A}$. First problem to be considered is the well-definedness of (\ref{CH}). As we mentioned in Remark \ref{rmk}, the spectrum of a transversally elliptic operator can be dense on $\R$ so we have to determine the order of the summation (\ref{CH}) in a natural way. We solve this by representation theory. By Corollary \ref{pwt3}, we have the following decomposition:
$$
L^2( E )=\overline{\bigoplus_{\alpha \in \N} [L^2(E)]_\alpha},
$$
where $[L^2(E)]_\alpha$ denotes the $(\tau_\alpha, L_\alpha)$-component in $L^2(E)$. Let $\{\psi_{\alpha,j}\}_{\alpha, j}$ be an orthonormal basis in $L^2(E)$ consisting of joint eigenfunctions of $A$, $\Delta_G$, $P$ such that $\{\psi_{\alpha,j}\}_{j}$ is an orthonormal basis in $[L^2(E)]_\alpha$ for each $\alpha$ (see Remark \ref{ONB3}). For each $\psi_{\alpha,j}$, we denote  the corresponding eigenvalues by $\widehat{\lambda_{A,\alpha,j}}$, $\widehat{\lambda_{G,\alpha,j}}$, $\widehat{\lambda_{P,\alpha,j}}$. Recall that the following relation holds (see Remark \ref{ONB4}):
\begin{eqnarray}\label{EJ}
\widehat{\lambda_{P,\alpha,j}}
=\widehat{\lambda_{A,\alpha,j}}^2+ \widehat{\lambda_{G,\alpha,j}}^m
=\widehat{\lambda_{A,\alpha,j}}^2+ \widetilde{\lambda_\alpha}^m,
\end{eqnarray}
where $\widetilde{\lambda_\alpha}$ is an eigenvalue of the Laplacian $\Delta:C^\infty(G)\rightarrow C^\infty(G)$ corresponding to $(\tau_\alpha, L_\alpha) \in \widehat{G}$. This identity implies, by discreteness of $\{\widehat{\lambda_{P,\alpha,j}}\}_{j}$, the set $\{\widehat{\lambda_{A,\alpha,j}}\}_{j}$ is also discrete and we can assume that 
$$
\widehat{\lambda_{A,\alpha,1}}\leq \widehat{\lambda_{A,\alpha,2}}\leq....
$$ 
Moreover, let also $\{{\lambda_{A,\alpha,k}}\}_k$ be the eigenvalues taken into account of their multiplicities, i.e., 
$$
{\lambda_{A,\alpha,1}}<{\lambda_{A,\alpha,2}}<....
$$
We put 
$$
W_{{\lambda_{A,\alpha,k}}}
=\{\varphi \in C^\infty( E ) | A\varphi={\lambda_{A,\alpha,k}}\ \varphi,\ \varphi\in H_\alpha\}.
$$
Notice that $W_{\lambda_{A,\alpha,k}}$ is finite dimensional since 
if $\phi \in W_{\lambda_{A,\alpha,k}}$ then $P\phi= (\lambda_{A,\alpha,k}^2 +\widetilde{\lambda_\alpha}^m)\phi$.
\begin{defi}
Suppose that the series 
\begin{eqnarray}\label{CH2}
\sum_{\alpha=1}^\infty \sum_{k=1}^\infty e^{-t{\lambda_{A,\alpha,k}}} \ ch\left(W_{{\lambda_{A,\alpha,k}}}\right)
\end{eqnarray}
converges weakly to a distribution on $G$, where $ch\left(W_{{\lambda_{A,\alpha,k}}}\right)$ denotes the character of a representation in the finite dimensional space $W_{{\lambda_{A,\alpha,k}}}$. We say this to be the character of $e^{-tA}$ and denote by $ch(e^{-tA})$.
\end{defi}
We put
\begin{eqnarray}\label{DDD}
f_\alpha(t)
=
\sum_{k=1}^\infty 
e^{-t{\lambda_{A,\alpha,k}}} \ 
n_\alpha(W_{{\lambda_{A,\alpha,k}}})
\ \ 
(\in [0,\infty]),
\end{eqnarray}
where
$n_\alpha(W_{{\lambda_{A,\alpha,k}}}) $ 
denotes the multiplicity of 
$(\tau_\alpha, L_\alpha) \in \widehat{G}$ 
in 
$W_{{\lambda_{A,\alpha,k}}}$. 
Then the following proposition states when the series (\ref{CH2}) converges weakly in $\mathcal{D}'(G)$.
\begin{prop}\label{Con}
Suppose that $f_\alpha(t)$ converges for each $\alpha$ and $t>0$,  and  there is a constant $C(t)$ which depends only on $t$ such that 
\begin{eqnarray}\label{ASS}
f_\alpha (t)
\leq
C(t)(1+\widetilde{\lambda_\alpha})^s.
\end{eqnarray}
Then the series (\ref{CH2}) converges weakly in $\mathcal{D}'(G)$. 
\end{prop}
\begin{proof}
If we decompose 
$ch
\left(
W_{{\lambda_{A,\alpha,k}}}
\right)
$
into the irreducible representation, the series ($\ref{CH2}$) can be written as 
\begin{eqnarray}\label{DCH}
\sum_{\alpha=1}^\infty 
\sum_{k=1}^\infty 
e^{-t{\lambda_{A,\alpha,k}}} \ 
ch
\left(
W_{{\lambda_{A,\alpha,k}}}
\right)
&=&
\sum_{\alpha=1}^\infty 
\sum_{k=1}^\infty 
e^{-t{\lambda_{A,\alpha,k}}} \ 
n_\alpha(W_{{\lambda_{A,\alpha,k}}}) 
\chi_\alpha
\\
&=&
\sum_{\alpha=1}^\infty 
f_\alpha(t)\chi_\alpha.
\end{eqnarray}
By Theorem $\ref{thm:1.1}$ (v) and the assumption (\ref{ASS}), the series ($\ref{CH2}$) converges weakly in $\mathcal{D}'(G)$, as claimed.
\end{proof}
\begin{prop}\label{example}
Suppose that there exists an oriented compact Riemannian manifold $N$, a Hermitian vector bundle $F$ over $N$, and an elliptic operator $P_N:C^\infty(F)\rightarrow C^\infty(F)$ satisfying the assumption in Lemma \ref{EO} such that  
$$
M=G\times N, \hsp
E=(G\times \mathbb{C}) \boxtimes F,\hsp
A=\id \otimes P_N,
$$
where $G$  acts on $M$ by
$$
g\cdot(h, q)=(gh,q),\hsp g\in G,\ (h,q)\in G\times N,
$$
and trivially on $F$. Then for $t>0$ the character of $e^{-tA}$ converges weakly in $\D'(G)$ and belongs to $L^2_{-k}(G)$ for $k>d$.
\end{prop}
\begin{proof}
If we denote by $0\leq{\lambda_{P_N,1}}<{\lambda_{P_N,2}}<...$ the eigenvalues of $P_N$, then 
\begin{eqnarray}\label{EV}
{\lambda_{P_N,k}}={\lambda_{A,\alpha, k}}
\end{eqnarray}
for all $k$ and $\alpha$. Moreover,  since by Corollary \ref{PWT} we can write 
$$
H_\alpha=(L_\alpha\otimes L_\alpha^*)\otimes L^2(F),
$$
the following relation holds:
\begin{eqnarray}\label{TE}
W_{{\lambda_{A,\alpha,k}}}
=
(L_\alpha\otimes L_\alpha^*)\otimes V_{{\lambda_{P_N,k}}}
=
\left[(\dim L_\alpha^* )L_\alpha\right]\otimes V_{{\lambda_{P_N,k}}}
\end{eqnarray}
Therefore, we have 
\begin{eqnarray}\label{BBB}
n_\alpha({W_{{\lambda_{A,\alpha,k}}}})
=
(\dim L_\alpha^*)\cdot(\dim V_{{\lambda_{P_N,k}}})
=d_\alpha (\dim V_{{\lambda_{P_N,k}}})
\end{eqnarray}
Hence by (\ref{EV}) and (\ref{BBB}), we have
\begin{eqnarray*}
f_\alpha(t)
&=&
\sum_{k=1}^\infty 
e^{-t{\lambda_{A,\alpha,k}}} \ 
n_\alpha(W_{{\lambda_{A,\alpha,k}}})
=
d_\alpha
\sum_{k=1}^\infty 
e^{-t{\lambda_{P_N,k}}} \ 
 (\dim V_{{\lambda_{P_N,k}}})
\\
&\leq&
d_\alpha
\sum_{k=1}^\infty 
e^{-t{\lambda_{P_N,k}}} \ 
\sum_{l=1}^k \dim V_{{\lambda_{P_N,l}}}
\leq
d_\alpha
\sum_{k=1}^\infty 
e^{-t{\widehat{\lambda_{P_N,k}}}} \ 
\sum_{l=1}^k \dim V_{{\lambda_{P_N,l}}}
\end{eqnarray*}
By the asymptotic formula (\ref{WAF}) and (\ref{WAF2}), there are constant $C>0$ and $C'(t)>0$ such that
\begin{eqnarray*}
f_\alpha(t)
&\leq&
C d_\alpha
\sum_{k=1}^\infty 
e^{-t{k^{\frac{2m}{n}}}} \ 
k^{\frac{n}{2m}}
=
C d_\alpha
\sum_{k=1}^\infty 
e^{-t{k^{\frac{2m}{n}}}} \ 
[{k^{\frac{2m}{n}}}]^{(\frac{n^2}{4m^2})}
\\
&\leq&
C d_\alpha
\sum_{k=1}^\infty 
e^{-t{k^{\frac{2m}{n}}}} \ 
[{k^{\frac{2m}{n}}}]^{({n^2})}
\leq
C
d_\alpha \int _1^\infty e^{-tx} x^{n^2} dx
=C'(t) d_\alpha.
\end{eqnarray*}
This shows,  by Proposition \ref{Con}, the series (\ref{CH2}) converges weakly in $\mathcal{D}'(G)$ to a distribution for $t>0$, which we denote by $ch(e^{-tA})$. 

We now investigate the smoothness of $ch(e^{-tA})$ in the sense of Sobolev.
By the above calculation, 
\begin{eqnarray}\label{CCC}
(1+\Delta)^{-s}{ch(e^{-tA})}
=
C'(t)\ \sum_{\alpha=1}^\infty \frac{d_\alpha}{(1+\widetilde{\lambda_\alpha})^s}\chi _\alpha.
\end{eqnarray}
By Proposition \ref{prop:1.1}, the series
\begin{eqnarray}
\sum_{\alpha=1}^\infty 
\left\|
\frac{d_\alpha}{(1+\widetilde{\lambda_\alpha})^s}\chi _\alpha
\right\|_{C (G)}
=\sum_{\alpha=1}^\infty \frac{d_\alpha^2}{(1+\widetilde{\lambda_\alpha})^s}
\end{eqnarray}
converges for $s>\frac{d}{2}$ . This shows that (\ref{CCC}) is continuous for $s>\frac{d}{2}$. Hence we can conclude that  $ch(e^{-tA})\in L^2_{-k}(G)$ for $k>d$. 
\end{proof}
Now we state the main result of this paper, which describes the  properties of $f_\alpha(t)$. 
\begin{thm}\label{Main}
Let $A:C^\infty( E )\rightarrow C^\infty( E )$ be a nonnegative formally self-adjoint transversally elliptic operator of order $m>0$. Then 
\begin{enumerate}
\item $f_\alpha(t)$ is finite for each $\alpha \in \N$ and $t>0$. More precisely, for each $\alpha$ and $\e>0$, there is a constant $C_5(\alpha,\e)>0$ depending only on $G$, $A$, $\alpha$, $\e$ such that for all $\alpha$ and $t>\e$
\beq \label{main}
f_\alpha(t)
\leq
C_5(\alpha,\e) 
+
C_1^m
\frac{\,n! \,   (1+\widetilde{\lambda_\alpha}^m)^{2m}} {(t-\e)^n}
e^{
t{\widetilde{\lambda_\alpha}}^{m/2}},
\eeq
where $n=\dim M$ and $C_1$ is a constant in the Weyl's formula (\ref{WAF}).
\item $f_\alpha(t)$ is smooth with respect to the parameter $t>0$.
\end{enumerate}

\end{thm}
\begin{rmk}
Since the conclusion in Theorem \ref{Main} does not satisfy the assumptions in Proposition \ref{Con}, we can not state that the heat operator of a nonnegative formally self-adjoint  transversally elliptic operator $A$ has the character $ch(e^{-tA})$ in general. 
\end{rmk}
To prove Theorem \ref{Main}, we first prove the following four lemmas.

\begin{lem}
\begin{eqnarray}\label{EST}
n_\alpha(W_{{\lambda_{A,\alpha,k}}})d_\alpha 
\leq
C_3({\lambda_{A,\alpha,k}}^2+ \widetilde{\lambda_\alpha}^m)^\frac{2m}{n},
\end{eqnarray}
where $C_3$ is a constant in (\ref{eq:2.12}).
\end{lem}
\begin{proof}
As in section 3, we set $V_{\lambda_A}:=\{\phi \in C^\infty(E)\ | \ A\phi = \lambda_A \phi \}$. Since then
$$
W_{{\lambda_{A,\alpha,k}}}
\subset
\overline{ 
\bigoplus
_{-{\lambda_{A,\alpha,k}}
\leq
\lambda_A
{\leq\lambda_{A,\alpha,k}}}
V_{\lambda_A}},
$$
we have
$$
n_\alpha
\left(
W_{{\lambda_{A,\alpha,k}}}
\right)
d_\alpha
\leq
n_\alpha
\left(
\overline{
\bigoplus
_{-{\lambda_{A,\alpha,k}}
\leq
\lambda_A
{\leq\lambda_{A,\alpha,k}}}
V_{\lambda_A}
}\right)
d_\alpha.
$$
By the inequality (\ref{eq:2.12}) in the previous section, we have (\ref{EST}), as claimed. 
\end{proof}
\begin{lem}\label{EFB}
Let $\{\psi_k\}_{k\in\N}$ be an orthonormal basis of $L^2(E)$ consisting of eigenfunctions of $P$ and we denote the corresponding eigenvalues by
$\widehat{\lambda_{P,1}}\leq\widehat{\lambda_{P,2}}\leq \cdots \leq \widehat{\lambda_{P,k}}\leq\cdots.
$
Then for each $\alpha$ and $k$, the following inequality holds:
$$
{\lambda_{A,\alpha,k}}
\geq
\begin{cases}
\sqrt{\widehat{\lambda_{P,k}}}
-
\sqrt{ \widetilde{\lambda_\alpha}^m} &(\widehat{\lambda_{P,k}}\geq  \widetilde{\lambda_\alpha}^m) \\
0 &(\mathrm{otherwise}).
\end{cases}
$$
\end{lem}
\begin{proof}
Since 
\begin{eqnarray}\label{EE}
\widehat{\lambda_{P,\alpha,k}}\geq \widehat{\lambda_{P,k}}
\end{eqnarray}
for all $k \in\mathbb{N}$, it follows from (\ref{EJ}) and (\ref{EE}) that
\begin{eqnarray*}
\widehat{\lambda_{A,\alpha,k}}^2
=
\widehat{\lambda_{P,\alpha,k}}-  \widetilde{\lambda_{\alpha}}^m
\geq
\widehat{\lambda_{P,k}}-  \widetilde{\lambda_{\alpha}}^m
\end{eqnarray*}
for each $k$ and $\alpha$. Since $\widehat{\lambda_{A,\alpha,k}}$ is nonnegative, we have 
$$
\lambda_{A,\alpha,k}
\geq
\widehat{\lambda_{A,\alpha,k}}
\geq
\begin{cases}
\sqrt{\widehat{\lambda_{P,k}}- \widetilde{\lambda_\alpha}^m}
\geq 
\sqrt{\widehat{\lambda_{P,k}}}
-
\sqrt{ \widetilde{\lambda_\alpha}^m} &(\widehat{\lambda_{P,k}}
\geq \widetilde{\lambda_\alpha}^m) \\
0 &(\mathrm{otherwise}),
\end{cases}
$$
as claimed.
\end{proof}
\begin{lem}\label{DOL}
For each $\e>0$, define $L_\e:= e^\frac{4m}{\e}$. Then
$
x^{4m} < e^{\e x }
$
for all $x>L_{\e}$. 
\end{lem}
\begin{proof}
Since $x>L_\e(>1)$, we have $\log x > \frac{4m}{\e}$. Then it follows from the fact $\log x < \sqrt{x}$ that 
$$
\frac{\log x}{x} < \frac{1}{\log x} < \frac{\e}{4m}.
$$
Therefore $\log x < \frac{\e}{4m} x$, which is equivalent to $x^{4m}< e^{\e x}$.
\end{proof}

\begin{lem}\label{DOK}
For each $\alpha$ and $\e>0$ we define 
$$K_{\alpha,\e}
:=
\min 
\left\{
k 
\ \left|\ 
\sqrt{\widehat{\lambda_{P,k}}}>\sqrt{  \widetilde{\lambda_\alpha}^m}+L_\e 
\right. 
\right\}.
$$
If $k\geq K_{\alpha,\e}$, then \ 
{\rm(i)}
$
\widehat{\lambda_{P,k}}
\geq 
\widetilde{\lambda_\alpha}^m
$
\  and \ 
{\rm(ii)}
$
\lambda_{A,\alpha,k}>L_{\e}.
$
\end{lem}
\begin{proof}
(i) Since 
$
\sqrt{\widehat{\lambda_{P,k}}}
>
\sqrt{  \widetilde{\lambda_\alpha}^m}+L_\e 
>
\sqrt{  \widetilde{\lambda_\alpha}^m},
$
we have 
$
\widehat{\lambda_{P,k}}
\geq 
\widetilde{\lambda_\alpha}^m.\\
$
(ii) By Lemma \ref{EFB} and (i), we have 
$
\lambda_{A,\alpha,k}
\geq\sqrt{\widehat{\lambda_{P,k}}}
-
\sqrt{ \widetilde{\lambda_\alpha}^m}>L_{\e}.
$
\end{proof}
\begin{proof}[Proof of Theorem \ref{Main}]
(i)
Using the inequality (\ref{EST}), we have
\begin{eqnarray*}
f_\alpha(t)
&=&
\sum_{k=1}^\infty 
e^{-t{\lambda_{A,\alpha,k}}} \ 
n_\alpha(W_{{\lambda_{A,\alpha,k}}})
\\
&\leq&
\sum_{k< K_{\alpha,\e}} 
e^{-t{\lambda_{A,\alpha,k}}} \ 
n_\alpha(W_{{\lambda_{A,\alpha,k}}})
+
\sum_{K_{\alpha,\e}\leq k}
e^{-t{\lambda_{A,\alpha,k}}} \ 
n_\alpha(W_{{\lambda_{A,\alpha,k}}})
\\
&\leq&
C_5(\alpha,\e)
+
\sum_{K_{\alpha,\e}\leq k}
e^{-t{\lambda_{A,\alpha,k}}} \ 
({\lambda_{A,\alpha,k}}^2+ \widetilde{\lambda_\alpha}^m)^\frac{2m}{n}, 
\eeqs
where 
$
C_5(\alpha,\e)
:=
\sum_{k< K_{\alpha,\e}} 
n_\alpha(W_{{\lambda_{A,\alpha,k}}}).
$
By Lemma \ref{DOK} (ii) we have  
$$
1< L_\e < \lambda_{A,\alpha,k} \leq \lambda_{A,\alpha,k}^2+\widetilde{\lambda_\alpha}^m 
$$
for all $k \geq K_{\alpha,\e}$. Therefore
\beqs
\sum_{K_{\alpha,\e}\leq k}
e^{-t{\lambda_{A,\alpha,k}}} \ 
({\lambda_{A,\alpha,k}}^2+ \widetilde{\lambda_\alpha}^m)^\frac{2m}{n}
&\leq&
\sum_{K_{\alpha,\e}\leq k} 
e^{-t{\lambda_{A,\alpha,k}}} \ 
({\lambda_{A,\alpha,k}}^2+ \widetilde{\lambda_\alpha}^m)^{2m}
\\
&=&
\sum_{K_{\alpha,\e}\leq k} 
e^{-t{\lambda_{A,\alpha,k}}} \ 
\sum_{r=0}^{2m}
\binom{2m}{r}
({\lambda_{A,\alpha,k}}^2)^r 
( \widetilde{\lambda_\alpha}^m)^{2m-r}
\\
&\leq&
\sum_{K_{\alpha,\e}\leq k} 
{\lambda_{A,\alpha,k}}^{4m}
e^{-t{\lambda_{A,\alpha,k}}} \ 
\sum_{r=0}^{2m}
\binom{2m}{r} 
( \widetilde{\lambda_\alpha}^m)^{2m-r}
\\
&=&
(1+\widetilde{\lambda_\alpha}^m)^{2m}
\sum_{K_{\alpha,\e}\leq k} 
e^{-t{\lambda_{A,\alpha,k}}} \ 
({\lambda_{A,\alpha,k}})^{4m}
\end{eqnarray*}
By Lemma \ref{DOK} (ii) and Lemma \ref{DOL}, we have 
\begin{eqnarray*}
\sum_{K_{\alpha,\e} \leq k}
e^{-t{\lambda_{A,\alpha,k}}} 
({\lambda_{A,\alpha,k}})^{4m}
&\leq&
\sum_{K_{\alpha,\e} \leq k}
e^{-t{\lambda_{A,\alpha,k}}} 
e^{\e{\lambda_{A,\alpha,k}}}
=
\sum_{K_{\alpha,\e} \leq k}
e^{-(t-\e){\lambda_{A,\alpha,k}}},
\end{eqnarray*}
By Lemma \ref{EFB} and Lemma \ref{DOK} (ii), we have
\begin{eqnarray*}
\sum_{K_{\alpha,\e} \leq k}
e^{-(t-\e){\lambda_{A,\alpha,k}}}
&\leq&
\sum_{K_{\alpha,\e} \leq k}
e^{-(t-\e)
\left(\sqrt{\widehat{\lambda_{P,k}}}
-
\sqrt{ \widetilde{\lambda_\alpha}^m}
\right)}
=
e^{(t-\e)
{\widetilde{\lambda_\alpha}}^{m/2}}
\sum_{K_{\alpha,\e} \leq k}
e^{-(t-\e)
\sqrt{\widehat{\lambda_{P,k}}}
}
\\
&\leq&
e^{
t{\widetilde{\lambda_\alpha}}^{m/2}}
\sum_{K_{\alpha,\e} \leq k}
e^{-(t-\e)
\sqrt{\widehat{\lambda_{P,k}}}
}
\leq
e^{
t{\widetilde{\lambda_\alpha}}^{m/2}}
\sum_{k=1}^\infty
e^{-(t-\e)
\sqrt{\widehat{\lambda_{P,k}}}
}.
\end{eqnarray*}
Moreover, by the asymptotic formula (\ref{WAF2}), we have
\begin{eqnarray*}
\sum_{k=1}^\infty
e^{-(t-\e)
\sqrt{\widehat{\lambda_{P,k}}}}
\leq
\sum_{k=1}^\infty
e^{-\sqrt{C_2}(t-\e)
k^{m/n}}
\leq
\sum_{k=1}^\infty
e^{-\sqrt{C_2}(t-\e)
k^{1/n}}
\leq
\int_0^\infty e^{-\sqrt{C_2} (t-\e) x^{1/n}} dx,
\end{eqnarray*}
If we put $l=x^{1/n}$, then we have
\begin{eqnarray*}
\int_0^\infty e^{-\sqrt{C_2} (t-\e) x^{1/n}} dx
&=&
\int_0^\infty e^{-\sqrt{C_2} (t-\e) l}\  \frac{dx}{dl} \ dl  
\\
&=&
n \int_0^\infty l^{n-1}\ e^{-\sqrt{C_2} (t-\e) l}\ dl
\\
&=&
\begin{cases}
\frac{n!}{(\sqrt{C_2})^n(t-\e)^n} & (t>\e) \\
\infty & (t \leq \e)
\end{cases}
\end{eqnarray*}
Hence if $t > \e$, we have 
\begin{eqnarray*}
f_\alpha(t)
&\leq&
C_5(\alpha,\e) 
+
\frac{n!(1+\widetilde{\lambda_\alpha}^m)^{2m}}{(\sqrt{C_2})^n(t-\e)^n}
e^{
t{\widetilde{\lambda_\alpha}}^{m/2}}.
\end{eqnarray*}
for each $\alpha$, and $\e>0$. Since $C_2=C_1^{-\frac{2m}{n}}$  (see Corollary \ref{WAF3}), we have (\ref{main}), as claimed.

(ii) To show that $f_\alpha(t)$ is smooth on $(0,\infty)$, it suffices to prove that $f_\alpha(t)$ is smooth on $[a,b]$ for each $0<a<b<\infty$. First we put
$$
g_{\alpha,k}(t)={e^{-t{\lambda_{A,\alpha,k}}} \ 
n_\alpha(W_{{\lambda_{A,\alpha,k}}})
}
$$
so that $f_\alpha(t)=\sum_{k} g_{\alpha,k}(t)$ and denote by
$$
g_{\alpha,k}^{(l)}(t)
=
(-\lambda_{A,\alpha,k})^l
{e^{-t{\lambda_{A,\alpha,k}}} \ 
n_\alpha(W_{{\lambda_{A,\alpha,k}}})
}
$$
its $l$-th derivative with respect to parameter $t$ for each $l\in\mathbb{N}$. Then we have
\begin{eqnarray}
\label{Maj}
\sum_{k=1}^\infty  \sup_{a\leq t\leq b}|g_{\alpha,k}^{(l)}(t)|
&=&
\sum_{k=1}^\infty  |g_{\alpha,k}^{(l)}(a)|
=
\sum_{k=1}^\infty  
({{\lambda_{A,\alpha,k}}})^l\ 
e^{-a{{\lambda_{A,\alpha,k}}}}\ 
n_\alpha(W_{{\lambda_{A,\alpha,k}}})
\end{eqnarray}
By (\ref{EST}), we have
\begin{eqnarray}
\sum_{k=1}^\infty  
({{\lambda_{A,\alpha,k}}})^l\ 
e^{-a{{\lambda_{A,\alpha,k}}}}\ 
n_\alpha(W_{{\lambda_{A,\alpha,k}}})
&\leq&
\nonumber
C_3
\sum_{k=1}^\infty  
{({\lambda_{A,\alpha,k}})}^l
e^{-a{{\lambda_{A,\alpha,k}}}}
({\lambda_{A,\alpha,k}}^2+  \widetilde{\lambda_\alpha}^m)^\frac{2m}{n}
\nonumber
\\
&\leq&
\nonumber
C_3
\sum_{k=1}^\infty  
{({\lambda_{A,\alpha,k}})}^l
e^{-a{{\lambda_{A,\alpha,k}}}}
(1+{\lambda_{A,\alpha,k}}^2+  \widetilde{\lambda_\alpha}^m)^\frac{2m}{n}
\\
&\leq&
\nonumber
C_3
\sum_{k=1}^\infty  
{({\lambda_{A,\alpha,k}})}^l
e^{-a{{\lambda_{A,\alpha,k}}}}
(1+{\lambda_{A,\alpha,k}}^2+  \widetilde{\lambda_\alpha}^m)^{2m}
\\
&\leq&
\label{MAJ}
C'(\alpha)
\sum_{k=1}^\infty
e^{-a{{\lambda_{A,\alpha,k}}}}
\ {(\lambda_{A,\alpha,k})}^{4m+l}.
\end{eqnarray}
In view of the proof of Theorem \ref{Main} $\rm{(i)}$, the series (\ref{MAJ}) converges. Hence the series 
$$
\sum_{k} g_{\alpha,k}^{(l)}(t)
$$
converges uniformly on $[a,b]$ for each $0<a<b<\infty$ and $l \in \N$. Thus $f_\alpha(t)$ is smooth on $(0,\infty)$.
\end{proof}


\end{document}